\documentclass[reqno,12pt]{amsart}

\usepackage{amssymb}
\usepackage[pdftex]{hyperref}
\hypersetup{colorlinks=true, linkcolor=black, citecolor=black}

\newcommand{\Z}{{\mathbb Z}}

\newcommand{\cA}{{\mathcal A}}

\newcommand{\alp}{\alpha}
\renewcommand{\phi}{\varphi}

\theoremstyle{plain}
\newtheorem{lemma}{Lemma}
\newtheorem{proposition}{Proposition}
\newtheorem{corollary}{Corollary}
\newtheorem{theorem}{Theorem}

\newtheorem{primetheorem}{Theorem}

\theoremstyle{remark}
\newtheorem{example}{Example}

\newcommand{\refc}[1]{~\ref{c:#1}}

\newcommand{\refl}[1]{~\ref{l:#1}}
\newcommand{\refp}[1]{~\ref{p:#1}}

\newcommand{\reft}[1]{~\ref{t:#1}}
\newcommand{\refx}[1]{~\ref{x:#1}}
\newcommand{\refb}[1]{~\cite{b:#1}}
\newcommand{\refe}[1]{\eqref{e:#1}}

\newcommand{\seq}{\subseteq}
\newcommand{\stm}{\setminus}
\newcommand{\est}{\varnothing}

\newcommand{\longc}{,\dotsc,}
\newcommand{\longp}{+\dotsb+}

\newcommand{\Zn}[1][n]{\Z_5^{#1}}

\author{Vsevolod F. Lev}
\email{seva@math.haifa.ac.il}
\address{Department of Mathematics, The University of Haifa at Oranim,
  Tivon 36006, Israel}
\title[Sum-free sets in $\Zn$]{Sum-free sets in $\Zn$}

\begin{document}
\baselineskip = 16pt

\begin{abstract}
It is well-known that for a prime $p\equiv 2\pmod 3$ and integer $n\ge 1$,
the maximum possible size of a sum-free subset of the elementary abelian
group $\Z_p^n$ is $\frac13\,(p+1)p^{n-1}$. We establish a matching
stability result in the case $p=5$: if $A\seq\Zn$ is a sum-free subset with
$|A|>\frac32\cdot5^{n-1}$, then there are a subgroup $H<\Z_5^n$ of size
$|H|=5^{n-1}$ and an element $e\notin H$ such that $A\seq(e+H)\cup(-e+H)$.
\end{abstract}

\maketitle

\section{Background and motivation.}

A subset $S$ of an abelian group is \emph{sum-free} if the equation $x+y=z$
has no solutions in the elements of $S$; that is, if $S$ is disjoint from
$2S$ where we use the standard notation $2S:=\{s_1+s_2\colon s_1,s_2\in S\}$.
The idea of a sum-free set goes back to Schur~\refb{s} who was motivated by
the modular version of the Fermat equation $x^n+y^n=z^n$. Despite this
initial motivation, sum-free sets are treated in \refb{s} as a combinatorial
object of independent interest. Originating from \refb{s}, the celebrated
\emph{Schur's theorem} (``the positive integers cannot be partitioned into
finitely many sum-free subsets'') is considered one of the origins of Ramsey
theory.

In the 1960's sum-free sets were studied under the name ``mutant sets''; see,
for instance,~\refb{ki}. The subject gained popularity when it turned out to
be related to a problem of Erd\H os. The reader is invited to check
\cite{b:gr,b:tv} for a historical account and further references.

How large can a sum-free subset of a given finite abelian group be? First
considered in 1968 by Diananda and Yap \cite{b:d,b:dy}, this basic question
did not receive a complete answer up until the year 2005 when it was
eventually resolved by Green and Ruzsa \refb{gr}.

Once the largest possible size is known, it is natural to investigate the
corresponding stability problem: what is the structure of sum-free subsets of
finite abelian groups of size close to the largest possible? In this respect,
the cyclic groups of infinite order and prime order, and elementary abelian
$p$-groups have received particular attention. Here we are concerned with the
groups of the latter type.

The case $p=2$ is of special interest due to its connections with the coding
theory and the theory of finite geometries, see \cite{b:cp,b:kl} for a
detailed explanation. Motivated by the applications in these areas, Davydov
and Tombak \refb{dt} established the structure of large sum-free subsets in
the binary settings. To state their principal result, we briefly review the
basic notions of periodicity and maximality.

The \emph{period} of a subset $A$ of an abelian group $G$ is the subgroup
$\pi(A):=\{g\in G\colon A+g=A\}\le G$; that is, $\pi(A)$ is the largest
subgroup $H\le G$ such that $A$ is a union of $H$-cosets. The set $A$ is
\emph{periodic} if $\pi(A)\ne\{0\}$ and \emph{aperiodic} otherwise. One also
says that $A$ is \emph{$H$-periodic} if $H\le\pi(A)$; that is, if $A$ is the
inverse image of a subset of the quotient group $G/H$ under the canonical
homomorphism $G\to G/H$.

A sum-free set is \emph{maximal} if it is not properly contained in another
sum-free set.

By $\Z_p^n$ we denote the elementary abelian $p$-group of rank $n$.

\begin{theorem}[{\cite[Theorem~1]{b:dt}}]\label{t:dtper} Let $n\ge 4$ and
suppose that $A\seq\Z_2^n$ is a maximal sum-free set. If $|A|>2^{n-2}+1$,
then $A$ is periodic.
\end{theorem}

From Theorem~\reft{dtper} it is not difficult to derive a detailed structural
characterization  of large sum-free sets in $\Z_2^n$.
\begin{primetheorem}[\cite{b:dt}]\label{t:dt}
Let $n\ge 4$ and suppose that $A\seq\Z_2^n$ is sum-free. If $|A|\ge
2^{n-2}+1$, then either $A$ is contained in a nonzero coset of a proper
subgroup, or there are an integer $k\in[4,n]$, a subgroup $H\le\Z_2^n$ of
size $|H|=2^{n-k}$, and a maximal sum-free subset
$\cA\seq\Z_2^n/H\simeq\Zn[k]$ of size $|\cA|=2^{k-2}+1$ such that $A$ is
contained in the inverse image of $\cA$ under the canonical homomorphism
$\Z_2^n\to\Z_2^n/H$.
\end{primetheorem}

As an easy consequence, we have the following corollary.
\begin{corollary}[\cite{b:dt}]\label{c:dt2}
Let $n\ge 4$ and suppose that $A\seq\Z_2^n$ is sum-free. If $|A|\ge 5\cdot
2^{n-4}+1$, then $A$ is contained in a nonzero coset of a proper subgroup.
\end{corollary}

Corollary~\refc{dt2} was independently obtained in~\refb{cp}.

In the ternary case, only an analog of Corollary~\refc{dt2} is known.
\begin{theorem}[\cite{b:l1}]\label{t:l3}
Let $n\ge 3$ and suppose that $A\seq\Z_3^n$ is sum-free. If $|A|\ge
5\cdot3^{n-3}+1$, then $A$ is contained in a nonzero coset of a proper
subgroup.
\end{theorem}
As shown in \cite{b:l1}, the bound $5\cdot3^{n-3}+1$ is sharp.

In this note, we study the first open case $p=5$ proving the following
result.
\begin{theorem}\label{t:main}
Let $n\ge 1$ and suppose that $A\seq\Zn$ is sum-free. If
$|A|>\frac32\cdot5^{n-1}$, then there are a proper subgroup $H<\Zn$ and an
element $e\notin H$ such that $A\seq(e+H)\cup(-e+H)$.
\end{theorem}

There are no reasons to believe that the assumption $|A|>\frac32\cdot5^{n-1}$
of Theorem~\reft{main} is sharp. On the other hand, it cannot be relaxed to
$|A|>5^{n-1}$.

\begin{example}\label{x:Ex1}
Suppose that $n\ge 3$ is an integer, and that $H<\Zn$ is a subgroup of index
$5$. Fix arbitrarily an element $e\notin H$ and a subset $S\seq H$ with
$S\cap(-S)=\est$ and $S\cup(-S)=H\stm\{0\}$, and let
$A:=(e+S)\cup\{2e,-2e\}\cup(-e-S)$. A straightforward verification shows that
$A$ is sum-free. Suppose now that $A$ is contained in a union of two cosets
of a subgroup $F<\Zn$. Since $A$ meets four $H$-cosets and just two
$F$-cosets, we have $F\ne H$. Furthermore, one of these $F$-cosets contains
at least half the elements of the set $e+S$. The intersection of this
$F$-coset with the coset $e+H$ has therefore size at least $\frac12
|S|=(|H|-1)/4>|H|/5$ while, on the other hand, the intersection of an
$H$-coset with an $F$-coset is a coset of a proper subgroup of $H$, and as
such, has size at most $|H|/5$, a contradiction showing that $A$ is not
contained in a union of two cosets of a proper subgroup.
\end{example}

We now turn to the proof of Theorem~\reft{main}.

\section{Proof of Theorem~\reft{main}}\label{s:proof}

Our argument is self-contained except that we need the following classical
result of Kneser (but see ~\cite[Theorem~6.1]{b:g} for our present
formulation).
\begin{theorem}[Kneser \cite{b:kn1,b:kn2}]\label{t:Kneser}
If $A_1\longc A_k$ are finite, nonempty subsets of an abelian group, then
letting $H:=\pi(A_1\longp A_k)$ we have
  $$ |A_1\longp A_k| \ge |A_1+H| \longp |A_k+H| - (k-1)|H|. $$
\end{theorem}
Theorem~\reft{Kneser} is referred to below as \emph{Kneser's theorem}.

We start with a series of ``general'' claims. At this stage, it is not
assumed that $A$ is a sum-free set satisfying the assumptions of
Theorem~\reft{main}.

\begin{lemma}\label{l:union}
Let $n\ge 1$ be an integer and suppose that $A\seq\Zn$ is sum-free. If
$|A|>\frac32\cdot5^{n-1}$ and $A$ is contained in a union of two cosets of a
proper subgroup $H<\Zn$, then there is an element $e\notin H$ such that
$A\seq(e+H)\cup(-e+H)$.
\end{lemma}

\begin{proof}
Since $2|H|\ge|A|>\frac32\cdot5^{n-1}$, we have $|H|=5^{n-1}$. Suppose that
$A=(e_1+A_1)\cup(e_2+A_2)$, where $A_1,A_2$ are contained in $H$, and
$e_1,e_2\in\Zn$ lie in distinct $H$-cosets. From $|A|>\frac32\cdot5^{n-1}$ we
get $|A_1|+|A_2|=|A|>\frac32\,|H|$. Therefore
$\min\{|A_1|,|A_2|\}>\frac12\,|H|$, and by the pigeonhole principle,
$2A_1=2A_2=A_1+A_2=H$. It follows that
$2A=(2e_1+H)\cup(e_1+e_2+H)\cup(2e_2+H)$. Since $A$ is sum-free, each of the
three cosets in the right-hand side is distinct from each of the cosets
$e_1+H$ and $e_2+H$, which is possible only if $e_2+H=-e_1+H\ne H$.
\end{proof}

By Lemma~\refl{union}, to prove Theorem~\reft{main} it suffices to show that
any sum-free set in $\Zn$ of size larger than $\frac32\cdot5^{n-1}$ is
contained in a union of two cosets of a proper subgroup.

\begin{proposition}\label{p:1}
Let $n\ge 1$ be an integer and suppose that $A\seq\Zn$ is sum-free. If
$|A|>\frac32\cdot5^{n-1}$, then $A$ cannot have non-empty intersections with
exactly three cosets of a maximal proper subgroup of $\Zn$.
\end{proposition}

\begin{proof}
The case $n=1$ is immediate. Assuming that $n\ge 2$, $A\seq\Zn$ is sum-free,
and $H<\Zn$ is a maximal proper subgroup such that $A$ intersects
non-trivially exactly three $H$-cosets, we obtain a contradiction.

Fix an element $e\in\Zn\stm H$, and for each $i\in[0,4]$ let $A_i:=(A-ie)\cap
H$; thus, $A=A_0\cup(e+A_1)\cup(2e+A_2)\cup(3e+A_3)\cup(4e+A_4)$ with exactly
three of the sets $A_i$ non-empty. Considering the actions of the
automorphisms of $\Zn[]$ on its two-element subsets (equivalently, passing
from $e$ to $2e,3e$, or $4e$, if necessary), we further assume that one of
the following holds:
\begin{itemize}
\item[(i)]   $A_2=A_3=\est$;
\item[(ii)]  $A_0=A_4=\est$;
\item[(iii)] $A_3=A_4=\est$.
\end{itemize}
We consider these three cases separately.

\paragraph{Case (i): $A_2=A_3=\est$}
In this case, $A=A_0\cup(e+A_1)\cup(4e+A_4)$, and since $A$ is sum-free, we
have $(A_1+A_4)\cap A_0=\est$. It follows that $|A_0|+|A_1+A_4|\le |H|$.
Consequently, letting $F:=\pi(A_1+A_4)$, we have $|H|\ge
|A_0|+|A_1|+|A_4|-|F|=|A|-|F|$ by Kneser's theorem. Observing that
$|F|\le\frac15|H|=5^{n-2}$, we conclude that
  $$ |A| \le |H|+|F| \le \frac65|H|=6\cdot 5^{n-2} < \frac32\cdot 5^{n-1}, $$
a contradiction.

\smallskip
\paragraph{Case (ii): $A_0=A_4=\est$}
In this case $A=(e+A_1)\cup(2e+A_2)\cup(3e+A_3)$ with $(A_1+A_2)\cap
A_3=\est$, and the proof can be completed as in Case (i).

\smallskip
\paragraph{Case (iii): $A_3=A_4=\est$}
In this case from $(A_0+A_1)\cap A_1=\est$, letting $F:=\pi(A_0+A_1)$, by
Kneser's theorem we get
  $$ |H| \ge |A_0+A_1|+|A_1| \ge |A_0|+2|A_1| - |F| $$
whence, in view of $|F|\le\frac15|H|$,
\begin{equation}\label{e:nov5a}
  2|A_1|+|A_0| \le \frac65\,|H|.
\end{equation}
Similarly, from $(A_0+A_2)\cap A_2=\est$ we get
\begin{equation}\label{e:nov5b}
  2|A_2|+|A_0| \le \frac65\,|H|.
\end{equation}
Averaging~\refe{nov5a} and~\refe{nov5b} we obtain
$|A|\le\frac65|H|<\frac32\cdot5^{n-1}$, a contradiction.
\end{proof}

\begin{proposition}\label{p:2}
Let $n\ge 1$ be an integer and suppose that $A\seq\Zn$ is sum-free, and that
$H<\Zn$ is a maximal proper subgroup. If there is an $H$-coset with more than
half of its elements contained in $A$, then $A$ has non-empty intersections
with at most three $H$-cosets.
\end{proposition}

\begin{proof}
Fix an element $e\in\Zn\stm H$, and for each $i\in[0,4]$ set $A_i:=(A-ie)\cap
H$; thus, $A=A_0\cup(e+A_1)\cup\dotsb\cup(4e+A_4)$. Suppose that
$|A_i|>0.5|H|$ for some $i\in[0,4]$. Since $2A_i=H$ by the pigeonhole
principle, we have $i>0$ (as otherwise $2A_0=H$ would not be disjoint from
$A_0$). Normalizing, we can assume that $i=1$. From $2A_1\cap A_2=\est$ we
now derive $A_2=\est$, and from $(A_1-A_1)\cap A_0=\est$ we get $A_0=\est$.
\end{proof}

In view of Lemma~\refl{union} and Propositions~\refp{1} and~\refp{2}, we can
assume that the set $A\seq\Zn$ of Theorem~\reft{main} contains fewer than
$\frac12\cdot 5^{n-1}$ elements in every coset of every maximal proper
subgroup.

\begin{lemma}\label{l:new}
Let $n\ge 1$ be an integer, and suppose that $A,B,C\seq\Zn$ satisfy
$(A+B)\cap C=\est$. If $\min\{|A|,|B|\}>2\cdot5^{n-1}$ and $C\ne\est$, then
$|A|+|B|+2|C|\le6\cdot5^{n-1}$.
\end{lemma}

\begin{proof}
Write $H:=\pi(A+B-C)$, and define $k\in[0,n]$ by $|H|=5^{n-k}$. We have
\begin{equation}\label{e:unnamed}
  \min\{|A+H|,|B+H|\} > 2\cdot 5^{n-1} = 2\cdot 5^{k-1}|H|
\end{equation}
while, by Kneser's theorem, and since $(A+B)\cap C=\est$ implies $0\notin
A+B-C$ and, consequently, $(A+B-C)\cap H=\est$,
\begin{equation}\label{e:ABCKneser}
  5^n-|H| \ge |A+B-C| \ge |A+H|+|B+H|+|C+H|-2|H|.
\end{equation}
Combining \refe{unnamed} and \refe{ABCKneser}, we obtain
  $$ 5^n \ge 2(2\cdot 5^{k-1}+1)|H| + |C+H| - |H|
                                    \ge 4\cdot 5^{k-1}|H| + |C+H| +|H|. $$
Consequently,
  $$ |C| \le |C+H| \le 5^{n-1}-|H|. $$
On the other hand, from~\refe{ABCKneser},
  $$ |A|+|B|+|C|\le 5^n+|H|. $$
Taking the sum of the last two estimates gives the result.
\end{proof}

\begin{proposition}\label{p:3}
Let $n\ge 1$ be an integer and suppose that $A\seq\Zn$ is a sum-free subset
of size $|A|>\frac32\cdot5^{n-1}$. If $H<\Zn$ is a maximal proper subgroup
such that every $H$-coset contains fewer than $\frac12|H|$ elements of $A$,
then there is at most one $H$-coset containing more than $\frac25|H|$
elements of $A$.
\end{proposition}

\begin{proof}
Suppose for a contradiction that there are two (or more) $H$-cosets that are
\emph{rich} meaning that they contain more than $\frac{2}{5}|H|$ elements of
$A$ each. Fix an element $e\in\Zn\stm H$ and write $A_i=(A-ie)\cap H$,
$i\in[0,4]$. Without loss of generality, either $A_0$ and $A_1$, or $A_1$ and
$A_2$, or $A_1$ and $A_4$ are rich.

If $A_0$ and $A_1$ are rich, then applying Lemma~\refl{new} with $H$ as the
underlying group, in view of $(A_0+A_1)\cap A_1=\est$ we get $4\cdot
\frac{2}{5}|H|<|A_0|+|A_1|+2|A_1|\le 6\cdot 5^{n-2}$, which is wrong.

If $A_1$ and $A_2$ are rich then, observing that $(A_1+A_1)\cap A_2=\est$, we
recover the contradictory
 $4\cdot \frac{2}{5}|H|<|A_1|+|A_1|+2|A_2|\le 6\cdot 5^{n-2}$.

Finally, if $A_1$ and $A_4$ are rich, then from
  $$ (A_1+A_4)\cap A_0 = (A_1+A_1)\cap A_2 = (A_4+A_4)\cap A_3 = \est $$
using Lemma~\refl{new} we obtain
\begin{align*}
  |A_1|+|A_4|+2|A_0| & \le 6\cdot 5^{n-2}, \\
  |A_1|+|A_1|+2|A_2| & \le 6\cdot 5^{n-2}, \\
  |A_4|+|A_4|+2|A_3| & \le 6\cdot 5^{n-2}.
\end{align*}
Taking the sum,
  $$ 3|A_1| + 3|A_4| + 2|A_0| + 2|A_2| + 2|A_3| \le 18\cdot 5^{n-2}; $$
that is, $2|A|+|A_1|+|A_4|\le 18\cdot 5^{n-2}$. However, from
$|A|>\frac32\cdot5^{n-1}$ and $\min\{|A_1|,|A_4|\}>\frac{2}{5}\cdot 5^{n-1}$
we derive $2|A|+|A_1|+|A_4|> 3\cdot 5^{n-1} + \frac45\cdot 5^{n-1}=19\cdot
5^{n-2}$, a contradiction.
\end{proof}

We use character sums to complete the argument and prove Theorem~\reft{main}.

\begin{proof}[Proof of Theorem~\reft{main}]
Suppose that $n\ge 2$, and that $A\seq\Zn$ is a sum-free set with
$\alp:=|A|/5^n>\frac3{10}$; we want to show that $A$ is contained in a union
of two cosets of a proper subgroup.

Denoting by $1_A$ the indicator function of $A$, consider the Fourier
coefficients
  $$ \hat1_A(\chi):=5^{-n}\sum_{a\in A} \chi(a),\ \chi\in\widehat{\Zn}. $$
Since $A$ is sum-free, we have $A\cap(A-A)=\est$, whence
  $$ \sum_\chi |\hat1_A(\chi)|^2\cdot \hat1_A(\chi) = 0; $$
consequently,
  $$ \sum_{\chi\ne 1}|\hat1_A(\chi)|^2\cdot \hat1_A(\chi) = -\alp^3 $$
and, as a result,
\begin{equation*}\label{e:Re}
  \sum_{\chi\ne 1}|\hat1_A(\chi)|^2\cdot \Re(\hat1_A(\chi)) = -\alp^3.
\end{equation*}
Comparing this to
  $$ \sum_{\chi\ne 1}|\hat1_A(\chi)|^2=\alp(1-\alp) $$
(which is an immediate corollary of the Parseval's identity), we obtain
  $$ \sum_{\chi\ne 1}|\hat1_A(\chi)|^2
                   \big((1-\alp)\,\Re(\hat1_A(\chi))+\alp^2\big) = 0. $$
We conclude that there exists a non-principal character
$\chi\in\widehat{\Zn}$ such that
\begin{equation}\label{e:Rsmall}
  \Re(\hat1_A(\chi)) \le -\frac{\alp^2}{1-\alp}.
\end{equation}

Let $F:=\ker\chi$, fix $e\in\Zn$ with $\chi(e)=\exp(2\pi i/5)$, and for each
$i\in[0,4]$, let $\alp_i:=|(A-ie)\cap F|/|F|$. By Propositions~\refp{1}
and~\refp{2},
we can assume that $\max\{\alp_i\colon i\in[0,4]\}<0.5$, and then by
Proposition~\refp{3} we can further assume that there is at most one index
$i\in[0,4]$ with $\alp_i>0.4$; that is, of the five inequalities
 $\alp_i\le 0.4\ (i\in[0,4])$, at most one fails, but holds true once the
inequality is relaxed to $\alp_i<0.5$. We show that this set of assumptions
is inconsistent with~\refe{Rsmall}. To this end, we notice that
  $$ 5 \Re(\hat1_A(\chi)) = \alp_0+s_1\cos(2\pi/5)+s_2\cos(4\pi/5) $$
where $s_1:=\alp_1+\alp_4$ and $s_2:=\alp_2+\alp_3\le 0.9$. Comparing
with~\refe{Rsmall}, we get
\begin{align*}
  -\frac{5\alp^2}{1-\alp}
    &\ge \alp_0+s_1\cos(2\pi/5)+s_2\cos(4\pi/5) \\
    &= \alp_0+s_1\cos(2\pi/5)+(s_2-0.9)\cos(4\pi/5) +0.9\cos(4\pi/5) \\
    &\ge \alp_0+s_1\cos(2\pi/5)+(s_2-0.9)\cos(2\pi/5) +0.9\cos(4\pi/5) \\
    &\ge \alp_0 + (5\alp-\alp_0-0.9)\cos(2\pi/5) +0.9\cos(4\pi/5) \\
    &\ge (5\alp-0.9)\cos(2\pi/5) +0.9\cos(4\pi/5),
\end{align*}
while the resulting inequality
  $$ -\frac{5\alp^2}{1-\alp} \ge (5\alp-0.9)\cos(2\pi/5) +0.9\cos(4\pi/5) $$
is easily seen to be wrong for all $\alp\in[0.3,1)$. This completes the proof
of Theorem~\reft{main}.
\end{proof}

\section*{Acknowledgment}
I am grateful to Leo Versteegen for the careful reading of the manuscript and
for spotting out a problem with the initial version of Example~\refx{Ex1}.

\medskip

\end{document}